\documentclass[10pt]{article}
\usepackage{amsmath,amssymb,amsthm, epsfig}
\usepackage{amsfonts}
\usepackage{amsmath}
\usepackage{amssymb}
\usepackage{color}
\usepackage[mathscr]{eucal}
\usepackage{mathptmx}
 \usepackage[usenames,dvipsnames]{pstricks}
\usepackage{epsfig}
 \usepackage{pst-grad} % For gradients
  \usepackage{pst-plot} % For axes

\title{\bf Geometric regularity estimates \\ for elliptic equations}

\author{ \textsc{Eduardo V. Teixeira} \\ \textit{\footnotesize Universidade Federal do Cear\'a}  \\ \textit{\footnotesize Fortaleza, CE, Brazil} }

\date{}

\def \div {\mathrm{div}}

\def \suchthat {\ \big | \ }

\newtheorem{theorem}{Theorem}
\newtheorem{lemma}[theorem]{Lemma}

\newtheorem{corollary}[theorem]{Corollary}

\theoremstyle{definition}

\theoremstyle{remark}

\numberwithin{equation}{section}

\newcommand{\intav}[1]{\mathchoice {\mathop{\vrule width 6pt height 3 pt depth  -2.5pt
\kern -8pt \intop}\nolimits_{\kern -6pt#1}} {\mathop{\vrule width
5pt height 3  pt depth -2.6pt \kern -6pt \intop}\nolimits_{#1}}
{\mathop{\vrule width 5pt height 3 pt depth -2.6pt \kern -6pt
\intop}\nolimits_{#1}} {\mathop{\vrule width 5pt height 3 pt depth
-2.6pt \kern -6pt \intop}\nolimits_{#1}}}

\begin{document}
\maketitle

%\begin{abstract}
%
%
%\noindent \textit{MSC:} 35J60, 35B65.
%
%\medskip
%
%\noindent \textbf{Keywords:}  Geometric regularity theory.
%
%{
%	\tableofcontents
%}
%
%\end{abstract}

\section{Overview} \label{sct intro}

Regularity theory for diffusive operators is among the finest treasures of the modern mathematical sciences. It appears in several different fields, such as, differential geometry, topology, numerical analysis, dynamical systems, mathematical physics, economics, etc.  Within the general field of Partial Differential Equations, regularity theory places itself in the very core, by bridging the notion of weak solutions (often found by energy methods or probabilistic interpretations) to the classical concept of solutions. 

\par
\medskip 

Probably the first contact undergraduate students have with regularity theory and its theoretical manifestations is in the complex calculus course. Along the first few lectures, the instructor defines the notion of holomorphic functions, which involves one complex derivative. Few lectures later, holomorphic functions are proven to be analytic, in particular of class $C^\infty$ in its domain of definition. What is even more surprising is the fact that one can bound {\it universally} all the derivatives of holomorphic functions, just by  controlling their $L^2$ norms. Yet at the beginning of that course, the students learn about the Cauchy-Riemann equation which provides a necessary and sufficient condition for a differentiable function $\mathbf{f}(x,y)$ to be holomorphic. A consequence of such an observation is that, at least in simply connect domains, a function $\mathbf{f}(x,y) = u(x,y) + i v(x,y)$ is holomorphic if, and only if, both $u$ and $v$ are harmonic functions, that is:
$\Delta u = \Delta v  =0$, where $\Delta$ denotes the Laplacian operator, $\Delta u := \partial_{xx} u + \partial_{yy} u$.

\par
\medskip 

Few courses down the road, the students finally meet the multi-dimensional version of the Laplacian operator in their introductory course on partial differential equations: if $u$ is twice differentiable in an open set of $\mathbb{R}^n$, then its Laplacian is defined as
\begin{equation}\label{Delta}
	\Delta u :=  \partial_{x_1x_1} u +  \partial_{x_2x_2} u + \cdots +  \partial_{x_nx_n} u,
\end{equation}
where $\partial_{x_kx_k}$, for $k=1,2,\cdots n$, denotes the pure second derivative in the $kth$-direction. Surprisingly enough, a similar regularity phenomenon, as in the complex calculus course, holds true. If $u$ is harmonic, in the sense that $u$ verifies $\Delta u = 0$, then $u$ is real analytic. Furthermore, one controls all derivatives of {\it any} harmonic function by the $L^2$ norm. That is, if $u$ is harmonic, say in $B_1 \subset \mathbb{R}^n$, then
$$
	\|D^{\alpha} u \|_{L^\infty(B_{1/2})} \le C_{n, |\alpha|} \cdot \|u\|_{L^2(B_1)},
$$
where $C_{n, |\alpha|}>0$ depends only on dimension $n$ and the order of the derivatives $|\alpha|$. 

\par
\medskip 

The Laplacian operator \eqref{Delta} appears in several mathematical models coming from very different backgrounds. From the physical perspective, such an operator represents diffusion. For sake of illustration, let us briefly revisit the theory of elastic membranes with forcing terms (wind, weight, gravitation, etc...). Physical considerations and some mathematical simplifications drive us to the following minimization problem:
\begin{equation}\label{Membrane}
	 \min \left \{ \int_{\Omega} \left (  |\nabla u|^2 + f(X)\cdot u \right ) dX \suchthat u = g \text{ on } \partial \Omega \right \},
\end{equation}
 where $\Omega$ is a bounded domain of $\mathbb{R}^n$, $g$ represents an eventual deformation (twist) of $\partial \Omega$ and $f$ is the forcing term. Now let $\varphi$ be your favorite smooth function that vanishes on $\partial \Omega$ and consider competing function $v_t(X):= u(X) + t \varphi(X)$. Since $u$ (the position of the membrane) is a minimum point of the functional \eqref{Membrane}, elementary Calculus  implies that
$$
	\dfrac{d}{dt}  \int_{\Omega} \left (  |\nabla v_t|^2 + f(X)\cdot v_t \right ) dX \Bigg |_{t=0} = 0,
$$ 
which readily reveals that $\int_{\Omega} \nabla u \cdot \nabla \varphi + f(X) \cdot \varphi dX = 0$. Integrating by parts and using the fact that $\varphi$ was taken arbitrarily, yields 
\begin{equation}\label{Poison}
	- \Delta u = f(X), \text{ in } \Omega,
\end{equation}
at least if $u$ has enough derivatives to justify the steps above. That is, the membrane satisfies a non-homogeneous Laplace equation. A key question in the theory of elastic membranes is to understand qualitatively the membrane deformation upon a forcing term $f$. Intuitively, the elasticity properties of the membrane absorbs the forces and adjusts itself in a more organized fashion, as to keep minimizing the energy considered (tension, say). In a bit more precise mathematical terms, one is interested in understanding the regularity of $u$ in terms of some lower order norms of $f$. 

\par
\medskip

Equation \eqref{Poison} is often called Poisson equation. Its regularity theory is part of a wider class of elliptic estimates called Schauder {\it a priori} estimates, which assure that solutions to a linear, uniformly elliptic equation with $C^{0,\theta}$ data, $0 < \theta < 1$, i.e., functions $u$ satisfying 
\begin{equation}\label{linear eq}
	 a_{ij}(X) D_{ij} u = f(X)
\end{equation}
where 
\begin{equation}\label{Holder assumptions}
	0< \lambda \le a_{ij}(X) \le \Lambda, \quad a_{ij}, f \in C^{0,\theta},
\end{equation}
are locally of class $C^{2,\theta}$.  Furthermore, there exists a constant $C>0$, depending only upon dimension, ellipticity constants $(\lambda, \Lambda)$, and the $\theta$--H\"older continuity of the data, $\|a_{ij}\|_{C^{0,\theta}}$ and $\|f\|_{C^{0,\theta}}$, such that
\begin{equation}\label{Schauder est}
	\|u\|_{C^{2,\theta}(B_{1/2})} \le C \cdot \|u\|_{L^\infty(B_1)}.
\end{equation}

This is a fundamental estimate in the theory of PDEs and its vast range of applications. There are also  Schauder {\it a priori} estimates for equations in divergence form:
\begin{equation}\label{linear eq div}
	 \div \left ( a_{ij}(X) \nabla  u \right ) = f(X),
\end{equation}
under the same assumptions \eqref{Holder assumptions}. Nonetheless, usually their regularity theory lacks one derivative: if $u$ solves \eqref{linear eq div}, under  \eqref{Holder assumptions}, then
\begin{equation}\label{Schauder est div}
	\|u\|_{C^{1,\theta}(B_{1/2})} \le C \cdot \|u\|_{L^\infty(B_1)}.
\end{equation}

Before continuing, one should notice that estimate \eqref{Schauder est} is by no means a trivial fact. Even if one looks at the Poisson equation $\Delta u = f(X)$, the Schauder estimate is saying that if the trace of the Hessian of function $u$ is $\theta$-H\"older continuous, for some $0< \theta < 1$, then all second order derivatives, even the ones that do not appear in the equation, like $\partial_{x_1x_2} u$ is also  $\theta$-H\"older continuous for the same exponent $\theta$. 

\par 
\medskip

Intuition becomes even less precise when one looks at problems modeled in heterogeneous media. The mathematical formulation of such problems now involve variable coefficients, $a_{ij}(X)$ as in \eqref{linear eq} and in \eqref{linear eq div}. In such context, it has been a common accepted aphorism that the continuity of the Hessian of a solution (in non-divergence form) or else the continuity of the gradient of a solution (in the divergence theory) could never be superior than the continuity of the medium. Notwithstanding, it has been recently established that such a phenomenon can occur, at least in particular meaningful points. 

\begin{theorem}\label{thm T}  Let $a_{ij}(X)$ be a $\theta$-H\"older continuous, uniform elliptic matrix, i.e., $0< \lambda \mbox {Id} \le a_{ij}(X) \le \Lambda \mbox{Id},$ with $a_{ij} \in C^{0,\theta}(B_1)$. If $u$ verifies
$$
	a_{ij}(X) D_{ij} u = 0, \text{ in } B_1
$$ 
and $Z$ is a zero Hessian point, i.e., $D^2u(Z) = 0$, then $u \in C^{2, 1^-}$ at $Z$. If, on the other hand, $u$ solves a divergence for equation
$$
	 \div \left ( a_{ij}(X) \nabla  u \right )  = 0, \text{ in } B_1
$$ 
and $Z$ is a critical point, i.e., $\nabla u(Z) = 0$, then  $u \in C^{1, 1^-}$ at $Z$.
\end{theorem}

The proof of Theorem \ref{thm T} is based on a rather powerful method  based on ``tangential regularity theories". We will comment about these set of ideas in Section \ref{sct TRT}  and will survey about the result from Theorem \ref{thm T} in Section \ref{sct 2}. 
\par 
\medskip

Regularity theory for elliptic problems in discontinuous media, i.e., when $a_{ij}(X)$ is merely elliptic, but with no further continuity assumptions, requires a whole new level of understanding and its mathematical treatment is  profound. Such an issue appears for instance in the study of problems coming from Calculus of Variations:
$$
	\int_\Omega \mathcal{F}(\nabla u) \ dX \longrightarrow \min
$$
where $\mathcal{F} \colon \mathbb{R}^n \to \mathbb{R}$ is a convex function (19th Hilbert's problem). It also appears in composite material sciences, where the diffusion process takes place within a medium made from two or more constituent materials. This causes  significant differences on physical and chemical properties along the medium. Each individual component of the material remains separate and distinct within the finished structure. From the macroscopic view point, we are led to the study of elliptic equations with coefficients given by:
$$
	b_{ij}(X) = \sum\limits_{\omega \in \mathcal{B}}\left [ a_{ij}(\omega) \cdot {1}_\omega \right ]  
$$
where $\mathcal{B}$ is a partition of $\Omega$, each $a_{ij}(\omega)$ is uniformly elliptic, i.e., $\lambda \text{Id} \le a_{ij}(\omega) \le \Lambda \text{Id}$,  and 
$a_{ij}(\omega) \not =  a_{ij}(\omega\prime), $ whenever $ \omega \not = \omega\prime.$

Another simple example of great interest in applied sciences are linear equations with homogenization process. It is given a smooth elliptic matrix $a_{ij}(X)$ and one needs to look at limiting configurations as $\varepsilon \to 0$ of solutions to elliptic equations (both in divergence and in non-divergence form) with coefficients given by
\begin{equation}\label{linear}
		 b_{ij}^\epsilon(X) = a_{ij}\left ( \frac{X}{\epsilon}  \right ). 
	\end{equation}
It is decisive then to obtain estimates that do not depend upon the homogenizing parameter $\epsilon>0$. Hence, it requires estimates that do not depend on continuity assumptions of the coefficient matrix.

\par 
\medskip

{\it Universal} continuity estimates for solutions to divergence form equations is the contents to the nowadays called De Giorgi-Nash-Moser regularity theory.
%\begin{theorem}[De Giorgi-Nash-Moser]\label{GNM} Let $u \in H^1(B_{1})$ be a weak solution to  
%$$
%	\div \left ( a_{ij}(X) \nabla  u \right )  = 0, \text{ in } B_1,
%$$
%with $a_{ij}$ satisfying $\lambda \text{Id} \le a_{ij}(X) \le \Lambda \text{Id}$, for $0 < \lambda \le \Lambda < \infty$.  Then, there exist constants $0< \alpha < 1$ and $C>1$, depending only upon dimension, $\lambda$ and $\Lambda$, such that
%$$
%	\|u\|_{C^{0,\alpha}(B_{1/2})} \le C \cdot \|u\|_{L^2(B_{1})}.
%$$
%\end{theorem}
Twenty years past until Krylov and Safonov established the non-divergence counterpart of such a result. 
%
%\begin{theorem}[Krylov-Safonov] \label{KS} Let $u$ satisfy  
%$$
%	 a_{ij}(X) D_{ij}  u  = 0, \text{ in } B_1,
%$$
%with $\lambda \text{Id} \le a_{ij}(X) \le \Lambda \text{Id}$, for $0 < \lambda \le \Lambda < \infty$.  Then, there exist  constants $0< \beta < 1$ and $C>1$, depending only upon dimension, $\lambda$ and $\Lambda$, such that
%$$
%	\|u\|_{C^{0,\beta}(B_{1/2})} \le C \cdot \|u\|_{L^\infty(B_{1})}.
%$$
%\end{theorem}
Through such eruditions, we now know that any solution to any elliptic equation has a universal modulus of continuity. In principle such results can be understood as {\it a priori} estimates. However, it is possible to establish the results using the language of weak solutions, i.e., $H^1$ distributional solutions in the divergence form equations and the theory of viscosity solutions for non-variational problems.

Even though these original results are for linear equations, they are indeed non-linear devices.  For instance, Krylov-Safonov Harnack inequality unlocks the theory of fully nonlinear elliptic equations. If $ \mbox{Sym}(n)$ denotes the space of $n\times n$ symmetric matrices, an operator $F\colon \mbox{Sym}(n) \to \mathbb{R}$ is said to be elliptic if 
$$
	\lambda \|P\| \le F(M+P) - F(M) \le \Lambda \|P\|,
$$
for all $P\ge 0$. If $u$ is a viscosity solution to a fully nonlinear elliptic equation
$$
	F(D^2u) = 0,
$$
then, {at least heuristically} both $u$ and $u_\nu$ satisfy a uniform elliptic equation, for any directional derivative $u_\nu$. Hence, by the Krylov-Safonov Harnack inequality, both $u$ and $u_\nu$ are of class  $C^{0,\beta}$ for some $0<\beta<1$. If in addition the operator $F$ is assumed to be convex, then a Theorem due to Evans  \cite{E} and Krylov  \cite{K} separately implies that $u$ is indeed $C^{2,\alpha}$. The question on whether solutions to fully nonlinear elliptic equations are of class $C^2$ challenged the community for more than twenty years. It was answered in the negative by  Nadirashvili and Vladut, \cite{NV1, NV2}, who exhibit solutions to uniform elliptic equations whose Hessian blows-up.  

As much as Krylov-Safonov regularity theory yields the study of fully nonlinear elliptic equations, the corresponding  De Giorig-Nash-Moser theorem was a starting point for the development of degenerate quasilinear elliptic equations of the $p$-Laplace type, back in late 1960's:
$$
	\Delta_p u := \div \left ( |\nabla u|^{p-2} \nabla u \right ).
$$

The $p$-Laplacian operator $\Delta_p$ appears naturally in the mathematical formulations of a number of physical problems. One can see it as a nonlinear version of the Laplacian (notice that when $p=2$ be get back the original Laplace operator). The smoothing effects of the $p$-Laplace operator is less efficient than its linear member, the Laplacian $\Delta$, but not because of nonlinear effects, but actually because the ellipticity of the operator degenerates as the gradient vanishes. It is known that alway from the set of critical points, $\mathfrak{C}(u):= \{X \suchthat \nabla u (X) = 0\}$, $p$-harmonic functions, i.e., solutions to 
\begin{equation}\label{p-harmonic}
	-\Delta_p u  = 0
\end{equation}  
are in fact quite smooth -- real analytic. That is, the {\it villain} of the theory is precisely the {\it a priori unknown} set of critical points. The first major result in the area is due to Uraltseva, who proved in \cite{U}, for the degenerate case, $p\ge 2$, that $p$-harmonic functions, are locally of class $C^{1,\alpha}$ for some exponent $0 < \alpha < 1$ that depends on $p$ and dimension. Nevertheless, $C_\text{loc}^{1,\alpha_p}$  is indeed optimal, since along its singular set $\mathfrak{C}(u)$, $p$-harmonic functions are not, in general, of class $C^2$, nor even $C^{1,1}$. The optimal exponent for gradient H\"older  continuity of $p$-harmonic functions are known only in dimension two, \cite{IM}.

\medskip

Nowadays, regularity theory for elliptic and parabolic problems is still a rather active line of research, that permeates several fields of mathematical sciences and applications. It is often regarded as a noble but difficult subject. It is present in a number of important recent breakthroughs as a decisive, key step. Through the next three Sections, we shall discuss a general geometric approach for regularity issues that emerge in several different contexts. In Section \ref{sct TRT} we will present a notion of tangential regularity theories, and in Sections \ref{sct 2} and \ref{sct DC} we will exemplify the powerful applicability of such a notion by establishing improved regularity estimates in two different scenarios. 

%%%%%%%%%%%%%%%%%%%%%%%%%%%%%%%%%%%%%

\section{Tangential regularity theories} \label{sct TRT}

In this Section we shall discuss about a general, unifying method for approaching regularity issues in models involving diffusion processes. Generally speaking, different models involve different intrinsic structures which have direct influence on the smoothing effects of the operators associated to the corresponding mathematical problem. On the very top podium in the regularity hierarchy among all diffusive differential operators is the rich theory of harmonic functions. 

\par

Infinite order {\it a priori} estimates, maximum principles, Hanack inequality, Liouville theorems, monotonicity and frequency formulas, unique continuation principles, energy estimates, etc, are some of the many mathematical tools available in the study of problems ruled by the Laplace operator. It is often that the model considered is not as generous as the Laplace equation, but still some regularity theory is available.

Tangential regularity theories refers to the heuristic idea of imagining the {\it universe} of all second order diffusive operators. Each model is represented by a dot. There are infinitely many paths between all the models, among themselves. Usually the paths represent some sort of compactness for solutions to problems around the models to be linked. Of course one can always construct straight  paths between two given operators, $L_0$ and $L_1$, namely $L_t := tL_1 + (1-t)L_0$. Depending on the aimed property to be proven for a given model, a specific path ought to be considered. 

Let us discuss a simple example as to clarify the ideas. Suppose we are willing to develop a regularity theory for the Poisson equation:
\begin{equation}\label{sct TRT eq01}
	\Delta u = f(X),
\end{equation}
discussed in Section \ref{sct intro}. Instead of looking at \eqref{sct TRT eq01} as a non-homogeneous Laplace equation, we should understand it as a family of models within the universe of all second order diffusive operators; for each {\it class of} forcing terms $f(X)$, we have one different model. The goal is to establish a regularity estimate for solutions to \eqref{sct TRT eq01} upon the control on a given norm of the forcing term, which determines the class of forcing terms considered. Such a control will be converted into compactness, and it will allow us to understand the Laplace equation, $\Delta h = 0$ as the tangential equation for the problem, in the following way: if the norm of $f$ is very tiny, then it means that the model is close (in that imaginary universe) to the harmonic functions representative. The program then is to import the rich {\it tangential} regularity theory (TRT) available for harmonic functions back to the original model, properly corrected through the path used to access that TRT.  For instance, suppose that $f \in L^p(B_1)$, with $n/2 < p < n$, where $n$ is the dimension. By performing a zoom in, i.e., by considering $v(X) = u(\rho X)$, we can assume that 
$\|f\|_p \ll 1$. Now, after such a  zoom-in, for each $\lambda >0$, we define:
$$
	u_\lambda (X) := \dfrac{1}{\lambda^{2-n/p}} u(\lambda X).
$$
Easily one verifies that
$$
	\Delta u_\lambda = f_\lambda(X) =: \lambda^{n/p} f(\lambda X),
$$
and that $\|f_\lambda\|_{L^p(B_1)} \le \|f\|_p \ll 1$. In other terms, we have verified that $u_\lambda$ is also a solution to a model that is close enough to the harmonic functions representative. Carrying this idea on, one proves that $u \in C^{0,\alpha}_\text{loc}(B_1)$ for $\alpha := 2-n/p$. Indeed, such a geometric approach to regularity extends in a natural way to establish the following sharp regularity estimate, see for instance \cite{T0, T1}, for more general results:

\begin{theorem} Let $u$ satisfy $\Delta u = f(X)$ in $B_1 \subset \mathbb{R}^n$. Suppose $f \in L^p_{{weak}}(B_1)$. Then
\begin{itemize}
	\item If $p = \frac{n}{2}$, then $\|u\|_{BMO(B_{1/2})} \le C_n \left ( \|u\|_{L^2(B_1)} + \|f\|_{L^{n/2}_{ {weak}}(B_1)} \right )$.
	\item If $n/2 < p < n$, then $\|u\|_{ C^{0,2-n/p}(B_{1/2})} \le C_n \left ( \|u\|_{L^2(B_1)} + \|f\|_{L^{p}_{ {weak}}(B_1)} \right ).$
	\item  If $p = {n}$, then $\|u\|_{LogLip(B_{1/2})} \le C_n \left ( \|u\|_{L^2(B_1)} + \|f\|_{L^{n}_{ {weak}}(B_1)} \right )$.
	\item If $ n < p < \infty$, then $\|u\|_{ C^{1,1-n/p}(B_{1/2})} \le C_n \left ( \|u\|_{L^2(B_1)} + \|f\|_{L^{p}_{ {weak}}(B_1)} \right ).$
	\item If $f \in BMO(B_1)$, then $\|u\|_{ C^{1,LogLip}(B_{1/2})} \le C_n \left ( \|u\|_{L^2(B_1)} + \|f\|_{BMO(B_1)} \right ).$
\end{itemize}	

\end{theorem}

In terms of the membrane problem illustrated in Section \ref{sct intro}, this means that even if we apply a singular force $f$ or order $|X|^{n/p}$, with $p>n/2$, the membrane absorbs and distributes such an impact in a way that it remains bounded. At the singular point, the membrane will develop a cusp of order $|X|^{2-n/p}$.  

It is possible to establish corresponding sharp parabolic estimates, both in space and in time, see for instance \cite{TU}. One can also obtain $L^q$ estimates of $u$ provided $f \in L^p$ for $p< n/2$. 

\par

Let us now discuss another type of results that can be accessed by similar reasonings. We revisit the theory of fully nonlinear equation operators. For any given pair of ellipticity constants $0< \lambda \le \Lambda < 1$, we shall interpret the equation $F(D^2u) = 0$ as a unified model. From the analytical point of view, a way to do that is to interpret the equation $F(D^2u ) = 0$ (for some $(\lambda, \Lambda)$-uniform elliptic operator) as
\begin{equation} \label{1}
	\mathscr{M}^{-}_{\lambda,\Lambda}(D^2u) \le 0 \le \mathscr{M}^{+}_{\lambda,\Lambda}(D^2u),
\end{equation}
and $\mathscr{M}^{\pm}_{\lambda,\Lambda}$ represent the extremal Pucci operators:
$$
	\mathscr{M}^{+}_{\lambda,\Lambda}(M) := \Lambda \left( \sum_{e_i >0} e_i \right) + \lambda \left( \sum_{e_i <0} e_i \right), \quad 
	\mathscr{M}^{-}_{\lambda,\Lambda}(M) := \lambda \left( \sum_{e_i >0} e_i \right) + \Lambda \left( \sum_{e_i <0} e_i \right)
$$
where $\{e_i : 1 \le i \le n\}$ are the eigenvalues of $M \in \mbox{Sym}(n)$. Easily one notice that if 

$$
	\frac{\lambda}{\Lambda} \to 1,
$$
then the model converges to the harmonic functions representative (i.e. the dot that represents the space of all harmonic functions). Thus, one can interpret the Laplace equation $\Delta u = 0$ as the {\it geometric tangential equation} of the {\it manifold} formed by fully nonlinear elliptic operators $F$ as  $\mathfrak{e} := 1- \frac{\lambda}{\Lambda}  \to 0$.  In turn, at every scale it is possible to find a harmonic function close to a viscosity solution to \eqref{1}, provided the ellipticity aperture is small enough. Iterating such an argument gives that the graph of a viscosity solution $u$ can be approximated by a quadratic polynomial  with an error of order $\sim \mbox{O}\left (\rho^{2+\alpha} \right )$, for any given $\alpha$. This proves the following Theorem (which can also be understood as a consequence of Cordes-Nirenberg estimates):

\begin{theorem}\label{principal}
Let $u \in C(B_1)$ be a bounded viscosity solution of \eqref{1}. Given $\alpha \in (0,1)$ there exists 
$\epsilon_0 = \epsilon_0(n,\alpha)>0$
such that, if $1-\frac{\lambda}{\Lambda} < \epsilon_0,$
then $u$ is locally of class $C^{2,\alpha}$ and 
\begin{equation} \label{pol1}
	\|u\|_{C^{2, \alpha}(B_{1/2})} \le C \cdot \|u\|_{L^\infty(B_1)},
\end{equation}
for a universal constant $C=C(n,\alpha)>0$.
\end{theorem}

Yet within the theory of fully nonlinear equations, we could find another path to access the tangential regularity theory of harmonic functions. Indeed, given a fully nonlinear elliptic operator $F$, we could look at the family of elliptic scalings 
$$
	F_\mu(M) := \dfrac{1}{\mu} F(\mu M), \quad \mu >0.
$$
This is a continuous family of operators preserving the ellipticity constants of the original equation. If $F$ is differentiable at the origin (recall, by normalization we always assume $F(0) = 0$), then 
$$
	F_\mu (M) \to \partial_{M_{ij}}F(0)M_{ij}, \quad \text{as } \mu \to 0.
$$
In other words, the linear operator  $M \mapsto \partial_{M_{ij}}F(0)M_{ij}$ is the \textit{tangential equation} of $F_\mu$ as $\mu \to 0$. Now, if $u$ solves an equation involving the original operator $F$, then $u_\mu := \frac{1}{\mu} u$ is a solution to a related equation for $F_\mu$. However, if in addition it is known that the norm of $u$ is at most $\mu$, then it accounts into saying that $u_\mu$ is a normalized solution to the $\mu$-related equation, and hence we can access the universal regularity theory available for the (linear) tangential equation by compactness methods. This proves the following Theorem (see \cite{S}, \cite{ASS}, \cite{DT}).

\begin{theorem} Let $u\in C^{0}\left( B_{1}\right) $ be a viscosity solution
to $F(D^{2}u) = f(X)$ in $B_1$, where $F$ is  a $C^1$, $(\lambda, \Lambda)$-elliptic operator and $f \in C^{0,\alpha}(B_1)$,
for some $0< \alpha < 1$. There
exist a $ \overline{\delta} > 0$, depending only upon $n,\lambda
,\Lambda, \alpha, \|f\|_\alpha$, and the $C^1$ norm of $F$, such that if
$$
    \sup_{B_{1}}| u | \leq \overline{\delta}
$$
then $u\in C^{2, \alpha } (B_{ {1}/{2}} )$ and
$$
    \|u\|_{C^{2,\alpha}(B_{1/2})}  \le M \cdot \overline{\delta},
$$
where $M$ depends only upon $n,\lambda,\Lambda, \alpha, \|f\|_\alpha$, and the $C^1$ norm of $F$.
\end{theorem}

\medskip

One can also try to ``import" regularity from the behavior of the operator at the infinity. That is, depending on how the operator behaves for very large matrices should give information of $C^{1,\alpha}$ estimates for solutions of the homogeneous problem $F(D^2u) = 0$. This brings to the notion of the {\it ressession} function:
$$
	F^\star(M) := \lim\limits_{\mu \to 0} \mu F(\mu^{-1} M).
$$
This notion also appears naturally in the study of free boundary problems governed by fully nonlinear operators, see \cite{RT, AT1}. The idea is that the family $F_\mu(M) :=  \mu F(\mu^{-1} M)$ 
forms a path of uniform elliptic operators (each $F_\mu$ is elliptic with the same ellipticity constants as $F$), jointing $F$ and $F^\star$. Should $F^\star$ have a better regularity theory, then one should be able to import it back to $F$, up to $C^{1,1^{-}}$. The following Theorem has been proven in \cite{ST}.

\begin{theorem} \label{thm main ST}
    Let $F$ be a uniform elliptic operator. Assume any recession function
    $$
        F^\star(M) := \lim\limits_{\mu\to 0} \mu F(\mu^{-1} M)
    $$
    has $C^{1,\alpha_0}$ estimates for solutions to the homogeneous equation
    $F^\star(D^2v) = 0$. Then, any viscosity solution to
    $$
        F (D^2u) = 0,
    $$
    is of class $C^{1, \min\{1,
    \alpha_0\}^{-}}_\text{loc}$. That is, $u \in C_\mathrm{loc}^{1,\alpha}$ for any
    $\alpha < \min\{1, \alpha_0\}$. In addition, there holds
    \begin{equation}\label{a priori est main thm}
    	\|u\|_{ C^{1,\alpha}(B_{1/2})} \le C \|u\|_{L^\infty(B_1)},
    \end{equation}
    for a constant $C>0$ that depends only on $n$, $\alpha$ and $F$.
\end{theorem}

An immediate Corollary of Theorem \ref{thm main} is the following:

\begin{corollary}\label{cor 01} Let $F \colon \mathcal{S}(n) \to \mathbb{R}$ be a uniform elliptic operator and $u$
a viscosity solution to $F(D^2u) = 0$ in $B_1$. Assume any
recession function $F^\star(M) := \lim\limits_{\mu \to 0} \mu
F(\mu^{-1}M)$ is concave. Then $u \in
C^{1,\alpha}_\text{loc}(B_1)$ for every $\alpha < 1$.
\end{corollary}

\medskip

As exemplified above, the method does not require the Laplace to be the tangential equation. For instance, one can study Poisson equations for more general operators, by interpreting the corresponding homogeneous problem as the tangential equation, \cite{C1, T0, T1, TU, ART}. Also, such an idea can be implemented in more adverse circumstances, such as in phase transmission problems, see for instance \cite{AT}.  Another tempting path within the universe of all diffusive models one can consider is the $p$ variation on the $p$-Laplacian operator. We conclude this Section explaining how the  set of ideas implies the continuity of the underlying regularity theory for $p-$Laplacian operators with respect to $p$.  This is done by the possibility of obtaining a universal compactness result. For instance, fix $M_0 \gg 2$ and work within the range $p \in [2, M_0]$. The following Lemma can be proven:

\begin{lemma}[Uniform in $p$ compactness] \label{unif compac}
Given $\delta >0$, there exists $\epsilon>0$, depending only on  $n$, $M_0$ and $\delta$, such that if $q\in [2, M_0]$, $u$ is a $q-$harmonic function in $B_1$, with $|u| \leq 1$, and $|q-p| < \epsilon$, then we can find a $p-$harmonic function $w$ in $B_{1/2}$, with $|w|\leq 1$, such that
\begin{equation} \label{closy}
	\sup_{B_{\frac{1}{2}}} |w-u|  \leq \delta.
\end{equation}
\end{lemma}

Let us comment on the proof of such a result: we suppose, for the sake of contradiction, that the thesis of the lemma does not hold true. This means that for a certain $\delta_0 >0$, there exist sequences $(q_j)_j$, $(u_j)_j$ and $(p_j)_j$, with
\begin{equation} \label{lad}
\left \{
\begin{array}{l}
q_j \in [2,M_0];\\
  \mathrm{div} \left( |\nabla u_j|^{q_j-2} \nabla u_j \right) = 0  \quad \textrm{in} \ \ B_1;\\
|u_j|\leq 1;\\
|p_j-q_j| \leq \frac{1}{j};
\end{array}
\right.
\end{equation}
however for every $p_j-$harmonic function $w$ in $B_{\frac{1}{2}}$,
\begin{equation} \label{hanninha}
\sup_{B_{\frac{1}{2}}} |u_j-w|  > \delta_0.
\end{equation}
By compactness, we have, up to subsequences, $p_j, q_j \rightarrow q_\infty \in [2, M_0]$, $u_j \rightarrow u_\infty$ locally in an appropriate functional space. By stability we can pass to the limit in the equation satisfied by the $u_j$ to conclude that $u_\infty$ is $q_\infty-$harmonic in $B_{\frac{2}{3}}$.We now solve, for each $p_j$, the following boundary value problem
\begin{equation} \label{bvp}
\left \{
\begin{array}{rcl}
  \mathrm{div} \left( |\nabla w_j|^{p_j-2} \nabla w_j \right) = 0  & \textrm{in}  & B_{\frac{2}{3}}\\
 w_j = u_\infty & \textrm{on}  & \partial B_{\frac{2}{3}}\\
\end{array}
\right.
\end{equation}
and pass to the limit in $j$, concluding that also $w_j \rightarrow u_\infty$ uniformly in $ B_{\frac{1}{2}}$ (by uniqueness). Finally, choosing $j$ large enough, we obtain
$$|u_j - w_j| \leq |u_j - u_\infty| + |w_j - u_\infty| \leq \frac{\delta_0}{2} + \frac{\delta_0}{2} = \delta_0 \quad \textrm{in} \ \ B_{\frac{1}{2}}$$
which is a contradiction to \eqref{hanninha}.
 
\medskip

One can use such a device to obtain improved sharp estimates  for problems governed by $p-$Laplacian operators, near the Laplacian, \textit{i.e.}, for $p$ close to 2.  After that, one can try to extend such estimates for a wider range, just by using now $(2+\epsilon)$-Laplacian as the tangential equation. Of course there is always the danger that at each iteration, the size of the step decreases  in a summable fashion. As to have a conclusive argument to cover the whole range of exponents $[2,M_0]$ as finer analysis is required.

\medskip

In the next two Section we shall apply these general ideas explained here as to establish improved regularity estimates for two different models. The readers will be able to recognize the steps illustrated by the examples above. Each problem to be treated involves different tangential paths. In Section \ref{sct 2} we shall access the $C^2$ {\it a priori} estimates for harmonic functions with zero first order approximations; whereas in Section \ref{sct DC} we will make use of maximum principle tools available for the geometric tangential equation of the model as to establish a surprising gain of smoothness for problems ruled by the $p$-Laplacian along a particularly important region.

%%%%%%%%%%%%%%%%%%%%%%%%%%%%%%%%%%%%%
%%%%%%%%%%%%%%%%%%%%%%%%%%%%%%%%%%%%%

\section{Improving Schaulder estimates} \label{sct 2}

In this Section we shall proof the divergence part of Theorem \ref{thm T}. In fact this is a consequence of a deeper Theorem established in \cite{T2}. The proof presented here though will be simplified and hence more appealing from the didactical point of view. 

\par

The initial set-up is the following, we have a function $u \in H^1(B_1)$ satisfying 
\begin{equation}\label{sct 2 eq1}
	 \div \left ( a_{ij}(X) \nabla  u \right ) = 0,
\end{equation}
in the distributional sense. The coefficients are assumed to be uniform elliptic and $\theta$-H\"older continuous, for some $0< \theta \ll 1$. From Schauder classical Theorem, $u$ is locally of class $C^{1,\theta}$, hence we can talk about $\nabla u$ at a given interior point. We will focus our analysis at a critical point of $u$, i.e., a point $Z\in B_1$, with $\nabla u(Z) = 0$. 

Let us open the floor by explaining the heuristics behind such a (surprising) result. If one is trying to show that an arbitrary  given function $u$ is of class $C^{1,\alpha}$ at, say, the origin, the task is to find an affine function $\ell(X)$ that approximates $u$ up to an error of order $\text{O}(r^{1+\alpha})$. If $0$ happens to be a critical point for $u$, i.e., $0\in \mathfrak{C}(u)$, then the 1st order of the approximation $\ell$ should be zero, and we are led to control the oscillation balance of $u$ around a real constant. Now this is in accordance to the general (still heuristic) fact that even though elliptic equations in divergence form are 2nd order differential operators, in fact it reflects an oscillation balance around constants, rather than affine functions, as in the non-divergence theory. That explains why in many situations the regularity theory for divergence form operators {\it has one derivative less} than the non-divergence one.

Continuing with the set-up, we notice that since $a_{ij}$ is continuous, up to a zoom-in, we can assume that $|a_{ij}(X) - a_{ij}(0)|$ is as small as we wish. Also, up to a change of variables, we can assume, with no loss, that $a_{ij}(0) = \delta_{ij}$ (the identity matrix). Here it is the first key observation:

\begin{lemma}\label{approx} Let $u \in H^1(B_1)$ be a weak solution to \eqref{sct 2 eq1}, normalized as to $\intav{B_1} |u|^2 dX \le 1$ and $\|a_{ij}\|_{C^{0,\theta}} \le 1$.  Then, given a number $\delta>0$, there exists a constant $\varepsilon > 0$, depending only on $\delta,  n,  \lambda, \Lambda$, such that if $\nabla u (0) = 0$ and
\begin{eqnarray}
	\left | a_{ij}(X) - \delta_{ij} \right | \le \varepsilon \label{small cond CL 3}
\end{eqnarray}
then there exists a harmonic function $h \colon B_{1/2} \to \mathbb{R}$, for which $0$ is also a critical point, such that
\begin{equation}\label{thesis - CL}
	  \intav{B_{1/2}} |u - h|^2 dX \le \delta^2.
\end{equation}
\end{lemma}

\begin{proof}
Let us assume, searching a contradiction, that the thesis of the Lemma fails. That means that there exist a $\delta_0 > 0$, a sequence $
 u_k \in H^1(B_1),$ with
 \begin{equation}\label{prof comp CM eq 00}
   \intav{B_1} |u_k(X)|^2 dX \le 1, \quad \& \quad   \nabla u_k (0)  = 0,
 \end{equation}
 for all $k\ge1$, and a sequence of elliptic $(\lambda, \Lambda)$-elliptic matrices $a^k_{ij}$ with 
\begin{equation}\label{prof comp CM eq 01.3}
 	\|a^k_{ij}\|_{C^{0,\theta}} \le 1, \quad \& \quad \left | a^k_{ij}(X) - \delta_{ij} \right | \to 0,
\end{equation}
such that 
\begin{equation}\label{prof comp CM eq 01}
             \div \left ( a^k_{ij}(X) \nabla u_k \right ) = 0 \text{ in } B_1;
\end{equation}
however
\begin{equation}\label{prof comp CM eq 03}
            \intav{B_{1/2}} |u_k(X) - h(X)|^2 dX \ge \delta_0^2, \quad \forall k \ge 1,
\end{equation}
for any harmonic function $h$  satisfying $\nabla h(0) = 0$. From classical Schauder estimates (or even simpler energy estimates of Caccioopoli type) we deduce that there exists a function $u_\infty \in H^1(B_{1/2})$ for which, up to a subsequence,
\begin{eqnarray}
    u_k &\rightharpoonup&  u_\infty \text{ in } H^1(B_{1/2}) \label{prof comp CM eq 04}  \\
    u_k &\to & u_\infty \text{ in } L^2(B_{1/2}) \label{prof comp CM eq 04.1}  \\
    D u_k(X) &\to& D u_\infty(X) \text{ pointwise for } X  \in B_{1/2} \label{ae conv grad},
\end{eqnarray}
Now, given a test function $\phi \in H^1_0(B_{1/2})$, in view of the approximation hypothesis \eqref{prof comp CM eq 01.3}, together with the limit granted in \eqref{prof comp CM eq 04}, we have
$$
    \begin{array}{lll}
        \displaystyle \int_{B_{1/2}} \langle  \nabla u_\infty, \nabla \phi \rangle dX &=& \displaystyle \int_{B_{1/2}} \langle a_{ij}^k(X)  \nabla u_k,  D \phi \rangle dX +\text{o}(1) \\
        & = & \text{o}(1),
    \end{array}
$$
as $k \to \infty$. Since $\phi$ was arbitrary, we conclude $u_\infty$ is a harmonic in $B_{1/2}$. Also, by \eqref{ae conv grad} together with assumption \eqref{prof comp CM eq 00}, we prove that $0$ is a critical point of $u_\infty$. Finally, confronting such conclusion  with \eqref{prof comp CM eq 04.1} and  \eqref{prof comp CM eq 03}, we reach a contradiction for $k \gg 1$. The Lemma is proven.
\end{proof}

Notice that Lemma \ref{approx} has a universal character in the sense that it is true for {\it any} solution of {\it any} elliptic equation that are under the assumptions of the Lemma -- the constant dependence is universal. This means that if we can rescale a given solution in a way to preserve the assumptions of Lemma \ref{approx}, then the same conclusion holds for the rescaled function. Continuing inductively, we should then be able to import the regularity theory available for harmonic function back to our original function, properly readjusted by the scaling process. This is our strategy from now on. 

\begin{lemma}\label{key lemma} Let $u \in H^1(B_1)$ be a normalized (i.e., $\intav{B_1} |u|^2 dX \le 1$) weak solution to \eqref{sct 2 eq1} with $\nabla u (0) = 0$, where $a_{ij}$ is $(\lambda, \Lambda)$-elliptic matrix with $\|a_{ij}\|_{C^{0,\theta}} \le 1$. Then, given $0< \alpha < 1$, there exist constants $0 < \varepsilon_0 < 1$, $0 < \varrho < \frac{1}{2}$, depending only upon $n, \lambda, \Lambda$ and $\alpha$, such that if
\begin{eqnarray}
	 	\left | a_{ij}(X) - \delta_{ij} \right | \le \varepsilon_0, \label{small cond key lemma 3}
\end{eqnarray}
then, we can find a universally bounded real constant $\tau \in \mathbb{R}$, i.e., $|\tau|< C(n, p, \lambda, \Lambda)$, such that
\begin{equation}\label{thesis key lemma}
	\intav{B_\varrho} | u(X) - \tau|^2 dX \le \varrho^{2(1+\alpha)}.
\end{equation}
\end{lemma}

\begin{proof}
For $\delta > 0$ to be chosen later, let $h$ be the harmonic function in $B_{1/2}$, with $\nabla h(0) = 0$,  that is $\delta$-close to $u$ in the $L^2$-norm. The existence of such a function has been granted by Lemma \ref{approx}. From $C^{2}$ estimates on $h$, there exists a constant $C_n$ depending only on dimension, such that
$$
    |h(X) - h(0)| \le C|X|^{2}.
$$
Since $\|h\|_{L^2} \le C$, by $L^\infty$ bounds,
$$
    |h(0)| \le C.
$$
We now estimate, for $\varrho >0$ to be adjusted {\it a posteriori},
$$
    \begin{array}{lll}
        \displaystyle \intav{B_{\varrho}} |u(X) - h(0)|^2 dX & \le &   2 \left (\displaystyle \intav{B_{\varrho}} |u(X) - h(X)|^2 dX +   \intav{B_{\varrho}} |h(X) - h(0)|^2 dX  \right ) \\
        & \le & 2 \delta^2 \varrho^{-n} + 2 C_n \varrho^{4}.
    \end{array}
$$
Since $0 < \alpha < 1$, it is possible to select $\varrho$ small enough as to assure
$$
     2 C_n \varrho^{2(1+\alpha)} \le \dfrac{1}{2} \varrho^{4}.
$$
Once selected $\varrho$, as indicated above, we set
$$
    \delta := \dfrac{1}{2} \varrho^{\frac{n}{2} + 1+ \alpha},
$$
which determines the smallness condition $\varepsilon_0$, in the statement of this Lemma, through the approximation Lemma \ref{approx}. The proof is concluded. 
\end{proof}

\medskip 

We now conclude the proof of $C^{1,1^-}$ regularity of solutions at critical points. Initially, if $u$ solves
$$
	\div \left ( a_{ij}(X) \nabla u \right ) = 0 \quad \text{ in } B_1,
$$
the rescaled function
$$
	v(X) := \frac{u(\gamma X)}{M},
$$
satisfies
$$
	\div \left ( a_{ij}(\gamma X) \nabla v \right ) = 0 \quad \text{ in } B_1, \quad \& \quad \|v\|_{L^2(B_1)} \le \frac{\gamma^{-n/2}}{M} \|u\|_{L^2(B_1)}.
$$
Now, if $a_{ij}$ is $C^{0,\theta}$ H\"older continuous at the origin and $a_{ij}(0) = \delta_{ij}$, we can choose $\gamma \ll 1$ so that
$$
	\|a_{ij}(\gamma X)\|_{C^{0,\theta}} \le 1 \quad \& \quad \left |a_{ij}(\gamma X) -\delta_{ij} \right | \le \varepsilon_0,
$$
where $ \varepsilon_0>0$ is the universal number from Lemma \ref{key lemma}. Afterwards, we can choose $M \gg 1$ so that $v$ is a normalized solution. That is, up to an scaling and a normalization, any solution to a divergence form elliptic equation with $\theta$-H\"older continuous coefficients can be framed into the assumptions required by Lemma  \ref{key lemma}. 

So, we start off the proof out from \eqref{thesis key lemma}. Next we rescale the function to the unit ball and  normalize it, that is, we define $u_1 \colon B_1 \to \mathbb{R}$, by
$$
	u_2(X) := \frac{u(\varrho X) - \tau}{ \varrho^{(1+\alpha)}}.
$$
Easily one verifies that $u_1$ is under the very same assumptions of Lemma \ref{key lemma}. Thus, there exists another universally bounded constant $\tilde{\tau} \in \mathbb{R}$, such that 
$$
	\intav{B_\varrho} | u_2(X) - \tilde{\tau}|^2 dX \le \varrho^{2(1+\alpha)},
$$
which rescaling back to $u$ gives
$$
	\intav{B_{\varrho^2}} | u(X) - {\tau_2} |^2 dX \le \varrho^{2(1+\alpha)},
$$
where 
$$
	\tau_2 - \tau = \tilde{\tau} \varrho^{1+\alpha}.
$$
In the sequel, we can define
$$
	u_3(X) := \frac{u(\varrho^2 X) - \tau_2}{ \varrho^{2(1+\alpha)}},
$$
and apply the argument again. Proceeding inductively, we obtain a sequence of real numbers $\tau_k$, satisfying
$$
	\left | \tau_{k+1} - \tau_k \right | \le C  \varrho^{k(1+\alpha)},
$$
such that 
$$
	\intav{B_{\varrho^k}} | u(X) - {\tau_k} |^2 dX \le \varrho^{k(1+\alpha)},
$$
In particular $\tau_k$ is a Cauchy sequence, hence it is convergent. From the inequality above, $\tau_k \to u(0)$. 

Finally, given any $0<r \ll 1$, let $k$ be the natural number that satisfies
\begin{equation}\label{proof main 09}
	 \varrho^{k+1} \le r < \varrho^k. 
\end{equation}
We estimate
\begin{equation}\label{proof main 10}
	\begin{array}{lll}
		\displaystyle \intav{B_r} |u(X) - u(0)|^2 dX &\le&   \left ( \frac{\varrho^k}{r} \right )^n  \displaystyle\intav{B_{\varrho^k}} |u(X) - u(0)|^2 dX \\
		&\le&   \frac{2}{\varrho^n} \displaystyle \left ( \intav{B_{\varrho^k}} |u(X) -\tau_k|^2 dX + |\tau_k - u(0)|^2 \right )\\
		&\le &   \frac{2 }{\varrho^n} \cdot \left (1 +\left ( \frac{C}{1-\varrho} \right )^2 \right ) \varrho^{k\cdot 2(1+\alpha)} \\
		& \le & \left [ \frac{2 }{\varrho^n} \cdot \left (1 +\left ( \frac{C}{1-\varrho} \right )^2 \right ) \cdot \frac{1}{\varrho^{2(1+\alpha)}} \right ] \cdot r^{2(1+\alpha)} \\
		& = & \tilde{C} \cdot r^{2(1+\alpha)},
	 \end{array}
\end{equation}
for a constant $\tilde{C} > 0$ that depends only upon universal parameters. The $C^{1,\alpha}$ regularity of $u$ at $0$ follows. Notice that $0 < \theta \ll \alpha < 1$ was taken arbitrary; hence we have in fact proven that $u$ is of class $C^{1,1^{-}}$ at a critical point.

\section{Improving regularity in $p$-dead core problems} \label{sct DC} 

In this Section we turn our attention towards reaction-diffusion equations governed by $p$-Laplace operator $(2\le p < \infty)$:
$$
	\Delta_p u = f(u).
$$

Of particular interest are models coming from porous catalysis or enzymatic processes. A decisive aspect of the mathematical formulation of such problems is the existence of dead-cores, i.e., regions where the density of the given substance vanishes identically. A prototype is the equation

\begin{equation}\label{p-DC}
	\Delta_p u = \lambda_0 u^q, 
\end{equation}
where $\lambda_0>0$ is the Thiele modulus, which adjusts the ratio of reaction rate to diffusion--convection rate and $0< q< p-1$. It is known, see \cite{PS}, that within this range of powers $q$, dead-core regions may appear. 

\par

Nonnegative solutions to \eqref{p-DC} presenting dead-core has been largely investigated. It follows by a general regularity theory for functions with bounded $p$-Laplacian that a bounded solutions to \eqref{p-DC} are locally of class $C^{1,\alpha}$ for an optimal exponent $\alpha < 1$ which depends on $p$ and on dimension. For instance the function $f(X) \mapsto |X|^{p/{p-1}}$ verifies $\Delta_p f =\mbox{cte.}$ for an easily computable constant. Notwithstanding this function is merely of class $C^{1, \frac{1}{p-1}}$ at zero. Even $p$-harmonic functions are merely $C^{1,\beta}$ for some $0< \beta < 1$. 

\par

The discussion above implies that there is no hope to established a better regularity than $C^{1,\alpha}$ for solutions to the dead-core problem \eqref{p-DC}. However, with the same spirit as to improve Schauder estimates, from Section \ref{sct 2}, we investigate here whether solutions to the $p$-dead core problem are more regular at some analytic meaningful points. After thinking for a bit on the problem, one should easily come to the conclusion that the most meaningful points in dead-core problems are the touching ground points, $\partial \{u> 0 \}$. We now state the key regularity improvement estimate we shall prove in this Section:

\begin{theorem}[Improved regularity] \label{IReg} Let $u$ be a nonnegative, bounded weak solution to \eqref{p-DC} in $B_1$ and let $\xi_0 \in \partial \{u> 0 \} \cap B_{1/2}$ be a touching ground point. Then for any point $X\in B_{1/2} \cap \{u>0\}$, there holds
$$
	u(X) \le C |X-\xi_0|^{\frac{p}{p - (q+1)}},
$$
for a constant $C$ depending only on dimension, $p$, $\|u\|_\infty$, and $q$. In particular solutions are locally of class $C^{{\frac{p}{p - (q+1)}}}$ at any touching ground point.
\end{theorem}
 
A warning should enhance the readers reaction towards Theorem \ref{IReg}:  the touching ground surface lies in the set of critical points of solutions; which are precisely the points where the diffusion of the operator degenerates.

It is then revealing to observe that, it follows in particular from Theorem \ref{IReg} that if dead-core exponent $q > \frac{p}{2} - 1$, then solutions are $C^2$ differentiable at any touching ground point. The proof will be based on a new flatness device, obtained by TRT methods, see \cite{T-DC}. 

\begin{lemma}[Flatness estimate] \label{FL} Given $\eta>0$, there exists a $\varrho = \varrho(\eta)>0$, depending only on $\eta$, dimension and p, such that if $\phi$ satisfies $0\le \phi \le 1$, $\phi(0) = 0$ and 
\begin{equation}\label{Eq01}
	 \Delta_p \phi - \delta^p \phi^q  = 0,
\end{equation}
with $0< \delta \le \varrho(\eta)$. Then, in $B_{1/2}$, $\phi$ is bounded by $\eta$, i.e.,
\begin{equation}\label{Eq02}
	 \sup\limits_{B_{1/2}} \phi \le \eta.
\end{equation}
\end{lemma}
\begin{proof}
	Indeed,  let us suppose, for the sake of contradiction,  that the thesis of the flatness estimate Lemma fails to hold. That means that for some $\eta_0 > 0$, there exists a sequence of functions $\phi_k$ satisfying $0\le \phi_k \le 1$, $\phi_k(0) = 0$ and
$$
	\Delta_p \phi_k = k^{-p} \phi_k^p, \quad \text{in } B_1;
$$
however
\begin{equation}\label{P03}
	 \sup\limits_{B_{1/2}} \phi_k \ge \eta_0, 
\end{equation}
for all $k\ge 1$. By the standard regularity theory for the $p$-Laplacian, , up to a subsequence, $\phi_k \to \phi_\infty$ in the $C^1$ topology. Clearly it verifies
\begin{eqnarray}
	 \phi_\infty &\ge& 0, \label{infty1} \\
	\phi_\infty(0) &=& 0. \label{infty2}\\
	\Delta_p \phi_\infty & = & 0, 
\end{eqnarray}
in $B_1$. It then follows by the strong maximum principle that $\phi_\infty \equiv 0$ in say $B_{2/3}$.  We reach a contradiction on \eqref{P03} by choosing $k \gg 1$ large enough.
\end{proof}

In the sequel, we shall apply Lemma \ref{FL} recursively in order to establish the improved regularity estimate along touching ground points. With no loss of generality we can assume $\xi_0 = 0$.  For positive constants $0< \kappa, ~\varrho_\star < 1$, to be determined {\it a posteriori}, consider the normalized, scaled function 
$$
	v(X) =  \kappa \cdot u( \varrho_\star X),
$$ 
defined in $B_1$.  Easily one verifies that $v$ solves
\begin{equation}\label{P01}
	\Delta_p v = \left (\kappa^{(p-1) - q} \cdot  \varrho_\star^p \right )v^q,
\end{equation}
in the distributional sense. Next we select
$$
	\kappa := \dfrac{1}{\|u\|_{L^\infty(B_1)}}, 
$$
so that $v$ is a normalized function in $B_1$. It is time to choose and fix the value of $\varrho_\star$. For that,  we select $\eta := 2^{-\frac{p}{p- (q+1)}}$ and Lemma \ref{FL} sponsors the existence of $\tilde{\varrho} > 0$, such that for any function $0\le \phi \le 1$, $\phi(0) = 0$, satisfying
$$
	 \Delta_p\phi - \delta^p  \phi^q  = 0,
$$
with $\delta \le \tilde{\varrho}$, we verify that
\begin{equation}\label{P04.1}
	\sup\limits_{B_{1/2}} \phi(X) \le 2^{-\frac{p}{p- (q+1)}}.
\end{equation}
With the value $\tilde{\varrho} > 0$ in hands, we decide that
$$
	\varrho_\star := \tilde{\varrho} \cdot \kappa^{\frac{- [p- (q+1)]}{p}}.
$$
With these choices, $v$ is under the assumptions of Lemma \ref{FL}, for $\eta =  2^{-\frac{p}{p-(q+1)}} $, i.e.,
$$
	\sup\limits_{B_{1/2}} v(X) \le 2^{-\frac{p}{p - (q+1)}}.
$$
In the sequel, let us define $v_2 \colon B_1 \to \mathbb{R}$, by
\begin{equation}\label{P04}
	v_2(X) :=  2^{\frac{p}{p- (q+1)}} \cdot v(\frac{X}{2}).
\end{equation}
One readily verifies that
\begin{eqnarray}
	&0 \le  v_2  \le 1; \label{v21} \\
	&v_2(0) = 0; \label{v22} \\
	&\Delta_p v_2 = \tilde{ \varrho}^p v_2^{q},
\end{eqnarray}
That is, $v_2$ falls under the conditions requested by the flatness estimate Lemma \ref{FL}, which implies, 
\begin{equation}\label{P05'}
	\sup\limits_{B_{1/2}} v_2(X) \le 2^{- \frac{p}{p-(q+1)}}.
\end{equation}
Rescaling \eqref{P05'} back to $v$ yields 
\begin{equation}\label{P05}
	\sup\limits_{B_{1/4}} v(X) \le 2^{-2 \cdot \frac{p}{p-(q+1)}}.
\end{equation}
Iterating inductively the above reasoning gives the following geometric decay:
\begin{equation}\label{P06}
	\sup\limits_{B_{\frac{1}{2^k}}} v(X) \le 2^{-k \cdot \frac{p}{p-(q+1)} }.
\end{equation}
Finally, given any $0<r< \frac{\varrho_\star}{2}$, let $k\ge 1$ be an integer such that $2^{-(k+1)} < \frac{r}{\varrho_\star} \le 2^{-k}$. We have
\begin{equation} \label{thm IReg final}
	\begin{array}{lll}
		\sup\limits_{B_r} u(X) &=& \sup\limits_{B_{\frac{r}{\varrho_\star} }} v(X) \\
		&\le&  \sup\limits_{B_{ 2^{-k}}} v(X) \\
		& \le & 2^{-k \cdot \frac{p}{p-(q+1)}} \\
		&\le & \left ( \dfrac{2}{\varrho_\star} \right )^{\frac{p}{p-(q+1)} } \cdot r^\frac{p}{p-(q+1)}\\
		& =& C(n,\lambda, \Lambda, \mu) \cdot \|u\|_{L^\infty(B_1)} \cdot r^\frac{p}{p-(q+1)},
	\end{array}
\end{equation}
and the Theorem is proven.

\bigskip
\bigskip

\noindent{\bf Ackwnoledgement.} This work has been partially supported by CNPq-Brazil, Capes-Brazil and Funcap. 
\bibliographystyle{amsplain, amsalpha}

\begin{thebibliography}{60}


\bibitem{AT} Amaral, M; Teixeira, Eduardo V. {\it Free transmission problems}.  To appear in Comm. Math. Phys.

\bibitem{ART} Ara\'ujo, D.; Ricarte, G.; Teixeira, E. {\it Geometric gradient estimates for solutions to degenerate elliptic equations}  To appear in Calc. Var. Partial Differential Equations.

\bibitem{AT1} Ara\'ujo, D.; Teixeira, E. {\it Geometric approach to nonvariational singular elliptic equations}. Arch. Rational Mech. Anal. {\bf 209} (2013), no 3, 1019--1054.

\bibitem{ASS} Armstrong S., Silvestre, L. and Smart, C. \textit{Partial regularity of solutions of fully nonlinear uniformly elliptic
equations}  Comm. Pure Appl. Math. {\bf 65} (2012), no. 8, 1169--1184.

\bibitem{C1} Caffarelli, Luis A. \textit{Interior a priori estimates for solutions of fully nonlinear equations.} Ann.
of Math. (2) \textbf{130} (1989), no. 1, 189--213.

\bibitem{CC} Caffarelli, Luis A.; Cabr\'e, Xavier
\textit{Fully nonlinear elliptic equations.} American Mathematical
Society Colloquium Publications, 43. American Mathematical Society,
Providence, RI, 1995.

%\bibitem{UG} Crandall, Michael G.; Ishii, Hitoshi; Lions, Pierre-Louis
%{\it User's guide to viscosity solutions of second order partial differential equations.}
%Bull. Amer. Math. Soc. (N.S.) {\bf 27} (1992), no. 1, 1--67.

\bibitem{DT} dos Prazeres, D.; Teixeira, Eduardo V. {\it Asymptotics and regularity for flat solutions to fully nonlinear elliptic problems}. To appear in Ann. Sc. Norm. Super. Pisa Cl. Sci.

\bibitem{E} Evans, Lawrence C.
{\it Classical solutions of fully nonlinear, convex, second-order elliptic equations.}
Comm. Pure Appl. Math. {\bf 35} (1982), no. 3, 333--363.

\bibitem{IM} Iwaniec, Tadeusz ; Manfredi, Juan J.
{\it Regularity of p-harmonic functions on the plane.}
Rev. Mat. Iberoamericana {\bf 5} (1989), no. 1-2, 1--19.

\bibitem{K} Krylov, N. V.
{\it Boundedly inhomogeneous elliptic and parabolic equations in a domain.} (Russian)
Izv. Akad. Nauk SSSR Ser. Mat. 47 (1983), no. 1, 75--108.
 
\bibitem{NV1} N. Nadirashvili and S. Vladut,
\textit{Nonclassical solutions of fully nonlinear elliptic equations}. Geom. Funct. Anal. {\bf 17} (2007), no. 4, 1283--1296.

\bibitem{NV2} N. Nadirashvili and S. Vladut,
{\it Singular viscosity solutions to fully nonlinear elliptic
equations.} J. Math. Pures Appl. (9) {\bf 89} (2008), no. 2,
107--113.

\bibitem{PS} P. Pucci and J. Serrin, {\it Dead cores and bursts for quasilinear singular elliptic equations}, SIAM J. Math. Anal., {\bf 38}  (2006) 259--278.

\bibitem{RT} G. Ricarte and E. Teixeira {\it Fully nonlinear singularly perturbed equations and asymptotic free boundaries.}  J. Funct. Anal. vol. 261, Issue 6, 2011, 1624--1673.

\bibitem{S} Savin, O. {\it Small perturbation solutions for elliptic equations.}
Comm. Partial Differential Equations {\bf 32} (2007), no. 4-6, 557--578.

\bibitem{ST} Silvestre, Luis and Teixeira, Eduardo V. {\it Regularity estimates for fully non linear elliptic equations which are asymptotically convex.} To appear in Progress in Nonlinear Differential Equations and Their Applications.
 
\bibitem{T0} Teixeira, Eduardo V. 
{\it Sharp regularity for general Poisson equations with borderline sources. } 
J. Math. Pures Appl. (9) {\bf 99} (2013), no. 2, 150--164. 

\bibitem{T1} Teixeira, Eduardo V. {\it Universal moduli of continuity for solutions to fully nonlinear elliptic equations.} Arch. Rational Mech. Anal. {\bf 211}, 3 (2014), 911--927.

\bibitem{T2} Teixeira, Eduardo V. {\it Regularity for quasilinear equations on degenerate singular sets.}  Math. Ann.  {\bf 358}, Issue 1 (2014), Page 241--256.

\bibitem{T3} Teixeira, Eduardo V. {\it Hessian continuity at degenerate points in nonvariational elliptic problems.} To appear in Int. Math. Res. Not. IMRN

\bibitem{TU} Teixeira, Eduardo V.; Urbano, Jos\'e Miguel {\it A geometric tangential approach to sharp regularity for degenerate evolution equations.} Anal. PDE. Vol. 7, No 3 (2014), 733--744.

\bibitem{T-DC} Teixeira, Eduardo V. {\it Regularity for the fully nonlinear dead-core problem.} To appear in Math. Ann.

\bibitem{U} Uraltseva, N.N. {\it Degenerate quasilinear elliptic systems.} Zap. Na. Sem. Leningrad. Otdel. Mat. Inst.
Steklov. (LOMI) 7, 184--222 (1968)


\end{thebibliography}

\bigskip
\bigskip

\noindent \textsc{Eduardo V. Teixeira} \\
\noindent Universidade Federal do Cear\'a \\
\noindent Departamento de Matem\'atica \\
\noindent Campus do Pici - Bloco 914, \\
\noindent Fortaleza, CE - Brazil 60.455-760 \\
 \noindent \texttt{teixeira@mat.ufc.br}

\end{document}